%% file: 2021-10-28-relmod.tex
\documentclass[12pt]{article}

\include{preamble}

\usepackage{epsfig}
\usepackage[utf8]{inputenc}

\newcommand\doi[2]{\href{http://dx.doi.org/#1}{#2}}

\newcommand{\Gr}{{\mathsf G}}
\newcommand{\kk}{\Bbbk}
\newcommand{\C}{\mathbb C}
\newcommand{\Z}{\mathbb Z}
\newcommand{\unit}{\mathds{1}}
\newcommand{\ideal}{\mathcal{I}} 
\newcommand{\mt}{\operatorname{\mathsf{t}}}
\newcommand{\rk}{\operatorname{rank}}

\newcommand{\ptr}{\operatorname{ptr}}
\newcommand{\qd}{\operatorname{\mathsf{d}}}
\newcommand{\catb}{\mathscr{B}} 
\newcommand{\Zt}{\ensuremath{\mathsf{Z}}}
\newcommand{\XX}{\ensuremath{\mathsf{X}}}
\newcommand{\wb}[1]{\overline{#1}}
\newcommand{\bp}[1]{\left( {#1}\right)}

\newcommand{\ms}[1]{\small \ensuremath{#1}}

\newcommand{\cat}{\mathcal{C}}
\newcommand{\epsh}[2]
         {\begin{array}{c} \hspace{-1.3mm}
        \raisebox{-4pt}{\epsfig{figure=#1,height=#2}}
        \hspace{-1.9mm}\end{array}}
\setlist[itemize]{parsep=0pt}
\setlist[enumerate]{parsep=0pt}

\author{Nathan Geer, Bertrand Patureau-Mirand, Matthew Rupert}
\begin{document}

\title{Some remarks on relative modular categories}
\maketitle

\abstract{We study properties of relative modular categories and derive sufficient conditions for their existence. In particular, we derive sufficient conditions for relative pre-modular categories to be non-degenerate and relative modular, and for the quotient of categories taking a particular form by their ideal of negligible modules to be generically semisimple.}

\tableofcontents

\section{Introduction} 

\setlength\parskip{1em}
\setlength\parindent{0pt}

In \cite{CGP14} Costantino and the first two authors constructed a class of non-semisimple quantum invariants of closed $3$-manifolds of Witten-Reshetikhin-Turaev type. They introduced therein the notion of a non-degenerate relative pre-modular category (under a slightly different name) which was used to construct these invariants. In \cite{D17}, Marco De Renzi found an additional condition which extended the $3$-manifold invariants of \cite{CGP14} to topological quantum field theories (see also \cite{BCGP16} for the case of $sl(2)$).  Relative pre-modular categories satisfying this condition are called relative modular categories (see Definition \ref{D:1} \& \ref{D:2}). Relative modular and pre-modular categories are non-finite non-semisimple categories satisfying numerous conditions on their structure. The theory of relative modular categories has been used to construct topological quantum field theories associated to various categories arising from unrolled quantum groups. The categories of weight modules for the unrolled quantum groups $U_{\xi}^H(\mathfrak{g})$ for $\mathfrak{g}$ any simple finite-dimensional complex Lie algebra and for $U_{\xi}^H\mathfrak{sl}(m|n)$ at odd root of unity $\xi$ are relative modular. We expect that these properties should hold for categories associated to other super unrolled quantum groups as well. 

Mikhaylov and Witten in \cite{MWi} consider super-group Chern-Simons theories (also see \cite{AGPS, Mik2015}).  An interesting questions is: ``As in the case of Witten's QFT interpretation of the Jones polynomial, are there quantum topology type invariants which are perturbative versions of the invariants discussed in \cite{MWi}?''  
Potential examples of such invariants are conjectured in Section 5 of \cite{AGP} and come from quotient categories of perturbative modules over quantum super Lie algebras.  
 However, the results of \cite{AGP} only give 3-manifold invariants in the case of quantum $\mathfrak{sl}(2|1)$ and 
 it is not known if these 3-manifold invariants extend to TQFTs.  
 The purpose of this paper is to generalize the techniques of \cite{AGP} and give sufficient conditions for when relative pre-modular category lead to 3-manifold invariants and TQFTs.    

The techniques of \cite{AGP} and sufficient conditions given in this paper are generalizations of standard techniques used in semi-simple categories.  However, non-semi-simple categories have vanishing quantum dimensions which make it necessary to make non-trivial modifications:  loosely speaking, the standard way to obtain a modular category from a quantum group is: 1) specialize $q$ to a root of unity; this forces some modules to have zero quantum trace, 2) quotient by \emph{negligible morphisms} of modules with zero quantum trace, 3) show the resulting category is finite and semi-simple. In \cite{AGP}, this construction is modified in the case of quantum super Lie algebras by taking the quotient by negligible morphisms corresponding to a non-zero modified trace (see \cite{GKP1}).  One of the main results of this paper is to give sufficient conditions for when this quotient construction gives 3-manifold invariants and TQFTs.  These conditions are described in terms of generalized $S$-matrices and M\"{u}ger centers.  

 In particular, in Section \ref{SS:ND} we derive sufficient non-degeneracy conditions on the modified $S$-matrix for a relative pre-modular category to be non-degenerate.  In Section \ref{relmod} we derive sufficient conditions (on the modified $S$-matrix and M\"{u}ger center)  for a relative pre-modular category to be relative modular.   In Section \ref{gens}, we determine sufficient conditions for when the quotient of a category by its ideal of negligible morphisms will be generically semisimple.  Finally, in Section \ref{SS:U} we prove that Conjecture 5 of \cite{AGP} is false and speculate on refinements of this conjecture.

\subsection{Results}
We now give a more precise summary of our results, for preliminary definitions see Section \ref{S:Preliminaries}. Throughout, non-degeneracy of a category refers to non-degeneracy as defined in Definition \ref{D:2} and an object in $\cat$ is called generic if it belongs to a subcategory $\cat_g$ with generic index $g \in \Gr \setminus \XX$ where $\Gr$ is the grading on $\cat$ and $\XX$ the small symmetric subset of $\Gr$ (see Definition \ref{smallsymm}). Let $S_g=(S_{ij})_{i,j\in I_g}$ be the modified $S$-matrix where $S_{ij}=\qd(V_j)S'_{ij}$ and $S'_{ij}$ is the scalar obtained by taking the partial trace of the double braiding with respect to $V_i,V_j$ (see Equation \ref{eq:Sprime}) and $\qd(V_j)$ the modified dimension of $V_j$.  Then our main results are (in text below this is Theorem \ref{T:Pre-modNonDeg} and \ref{degenmod}, respectively):

\begin{theorem}
  A pre-modular $\Gr$-category $\mcC$ is non-degenerate if there
  exists $g \in \Gr \setminus \XX$ such that the corresponding
  $S$-matrices $S_g$ and $S_{-g,g}$ are non-degenerate.
\end{theorem}

\begin{theorem}\label{degenmodintro}
  Let $\mcC$ be a relative pre-modular category.
\begin{enumerate}
\item If $\mcC$ is relative modular, then $S_g$ is non-degenerate for some generic index $g \in \Gr \setminus \XX$.
\item If $\mcC$ is unimodular and $S_g$ is non-degenerate for some generic index $g \in \Gr \setminus \XX$ and $\cat_0$ has $\Zt$-trivial Müger center, then $\mcC$ is relative modular.
\end{enumerate}
\end{theorem}
The following proposition (in the text this is Proposition \ref{DMug}) helps check the assumptions of Theorem~\ref{degenmodintro}:
\begin{proposition}
  Suppose that $\cat$ is abelian, $\{\sigma(k)\}_{k\in\Zt}$ has no self
  extension and that $\cat_0$ has $N$ simple
  $\Zt$-orbits. If for some $g\in\Gr\setminus\XX$, we have: 
\begin{enumerate}
\item $S_g$ is a $N\times N$ invertible matrix.
\item Both $S_g$ and $S_{-g,g}$ have a row with no zero entry.
\end{enumerate}
Then the Müger center of $\cat_0$ is $\Zt$-trivial. 
\end{proposition}

To construct relative and pre-relative modular categories from a category $\mathcal{C}$ admitting a modified trace we consider the full subcategory generated by some subset $S \subset \mathrm{Ob}(\mathcal{C})$ with respect to tensor products, direct sums, and retracts. This category is denoted by $\mathcal{C}^S$; we considers the quotient $\bar{\mathcal{C}}^S$ of this category by its ideal of negligible morphisms with respect to the modified trace. This method is used in \cite{AG,AGP} to construct (pre)-relative modular categories. In particular, this technique is applied to the category $\mathcal{D}^{\wp}$ of perturbative $U_{\xi}^H\mathfrak{sl}(m|n)$-modules to construct a category $\mathcal{D}^{\aleph}$ (see Subsection \ref{SS:slmn} for a precise construction). It was shown in \cite{AGP} that  $\mathcal{D}^{\aleph}$ carries the structure of a relative pre-modular category and non-degeneracy was proven for $\mathfrak{sl}(2|1)$.  Stated therein is the following conjecture:
\begin{conjecture}\label{introconj}
$\mathcal{D}^{\aleph}$ is a relative modular category.
\end{conjecture}

A careful analysis of the matrix $S_g$ for $\mathcal{D}^{\aleph}$ in Section \ref{SS:U} combined with Theorem \ref{degenmodintro} leads us to the following result ( Corollary \ref{rankintro} below)
\begin{corollary}
Conjecture \ref{introconj} is false.
\end{corollary}
We speculate on possible refinements of this conjecture in Section \ref{SS:U}. In Section \ref{gens}, we derive results categorifying the arguments given in \cite[Section 5]{AGP} to prove that $\mathcal{D}^{\aleph}$ is generically semi-simple. We expect that these results will find applications in proving generic semi-simplicity of categories constructed using techniques similar to those used for $\mathcal{D}^{\aleph}$. We have the following result (see Definitions \ref{D:WS} and \ref{D:CS}) which appears as Theorem \ref{sd} below:

\begin{theorem}
Let $\mcC$ be a linear ribbon category and $S \subset \mathrm{Ob}(\mcC)$. Then,
\begin{enumerate}
\item[1)] If all objects in $S$ and all tensor products of objects in $S$ have strong decomposition, then all objects in $\cat^S$ have strong decomposition.
\item[2)] Suppose $\mcC$ is graded by a group $G$ with small symmetric subset $\XX$. If all generic objects in $S$ and all generic tensor products of objects in $S$ have strong decomposition, then all generic objects in $\mcC^S$ have strong decomposition.
\end{enumerate}
\end{theorem}
The quotient of any category satisfying $(1)$ of this theorem by its ideal of negligible morphisms will be semisimple and any category satisfying $(2)$ will be generically semisimple. Examples of such categories often admit a distinguished object $\mathsf{v} \in \mathcal{C}$ which makes checking the conditions of this theorem easier (Corollaries \ref{Cor1} and \ref{Cor2} below):

\begin{corollary}\label{introCor1}
Let $\mcC$ be a linear ribbon category and let $S^{\mathsf{v}}:=S \cup \{\mathsf{v}\}$ where $S \subset \mathrm{Ob}(\cat)$ consists of objects with strong decomposition and $\mathsf{v} \in \mathrm{Ob}(\cat)$ is a distinguished object with strong decomposition such that:
\begin{enumerate}
\item If $V \in S$ then $V \otimes \mathsf{v}^n$ has strong decomposition for all $n \in \mathbb{Z}_{\geq 0}$.
\item For any $U_1,U_2 \in S$, $U_1 \otimes U_2 \cong \bigoplus \limits_k V_k$ where $V_k$ is a retract of $\tilde{V}_k \otimes \mathsf{v}^{n_k}$ for some $\tilde{V}_k \in S$ and $n_k \in \mathbb{Z}_{\geq 0}$.
\end{enumerate}
Then, every object in $\cat^{S^{\mathsf{v}}}$ has strong decomposition.\\
\end{corollary}

\begin{corollary}\label{introCor2}
Let $\mcC$ be a linear ribbon category graded by a group $\Gr$ with a small symmetric subset $\XX$.  As above, let $S^{\mathsf{v}}:=S \cup \{\mathsf{v}\}$ where $S \subset \mathrm{Ob}(\cat)$ consists of objects with strong decomposition and $\mathsf{v} \in \mathrm{Ob}(\cat)$ is a distinguished object with strong decomposition.  Suppose  that the following three conditions hold:
\begin{enumerate}
\item Every generic object in $S^{\mathsf{v}}$ has strong decomposition.
\item Given any $V \in S$ and $n \in \mathbb{Z}_{\geq 0 }$, if $V \otimes \mathsf{v}^n$ is generic, then it has strong decomposition.
\item For $U_1 ,U_2 \in S$, if $U_1 \otimes U_2$ is generic, then $U_1 \otimes U_2 \cong \bigoplus\limits_k V_k$ where $V_k$ is a retract of $\tilde{V}_k \otimes \mathsf{v}^{n_k}$ for some $\tilde{V}_k \in S$ and $n_k \in \mathbb{Z}_{\geq 0}$.
\end{enumerate}
Then, every generic object in $\mathcal{C}^{S^{\mathsf{v}}}$ has strong decomposition.\\
\end{corollary}
Corollary \ref{introCor2} applies in particular to the construction of $\mathcal{D}^{\aleph}$ and provides a categorification of the arguments used to prove generic semisimplicity of $\mathcal{D}^{\aleph}$ in \cite{AGP}.

\section{Preliminaries}\label{S:Preliminaries}
\subsection{Relative (pre-)modular category}\label{SS:RelModDef}
In this section we recall briefly the main theoretical notion used to
define TQFTs as in \cite{CGP14} and \cite{D17}. Let $\Gr$ be a commutative group.  The notion of a relative
$\Gr$-modular category appeared in \cite{CGP14}, in order to build
invariants of decorated 3-manifolds (where the decoration includes a
$\Gr$-valued $1$-cohomology class).  We should call these categories
non-degenerate relative pre-modular as De Renzi gave in \cite{D17} an
additional ``modularity condition'' that ensures the existence of an
underlying $1+1+1$-TQFT.

The construction is based on the notion of modified trace (or m-trace
for short) which replaces the usual categorical trace in modular
categories.  Let $\mcC$ be a linear ribbon category over a field
$\kk$ (a ribbon category with unit $\unit$ where the hom-sets are
$\kk$ vector spaces, the composition and tensor product of morphisms
are $\kk$-bilinear, and the canonical $\kk$-algebra map
$\kk \to \End_\mcC(\unit), k \mapsto k \, \Id_\unit$ is an
isomorphism).  An object $V\in\mcC$ is simple if
$\End_\mcC(V)=\kk\Id_V$.  It is regular if its evaluation is an
epimorphism.
\begin{definition}\label{mtrace}[\cite{GKP1}]\ \\
\noindent a) An {\em ideal} $\ideal$ in 
$\mcC$ is a full subcategory which satisfies:
\begin{enumerate}
\item Stable under retracts: if $W$ in $\ideal$, $V$ is an object of $\mcC$ and there exist $\alpha:V\to W$ and $\beta:W\to V$ morphisms such that $\beta\alpha=\Id_{V}$ then $V\in \ideal$.

\item Absorbs tensor products: if $U\in\ideal$ then for all
  $V \in \mcC$, $U\otimes V\in\ideal$.
\end{enumerate}
b) An \emph{m-trace} on an ideal $\ideal$ is a family of linear functions 
$\{\mt_V\}_{V\in\ideal}$ where $\mt_V:\End_\mcC(V)\to\kk$ satisfies:
\begin{enumerate}
 \item Cyclicity property: $\mt_V(fg)=\mt_U(gf)$, 
\item  Partial trace
property: 
$\mt_{U\otimes V}(f)=\mt_U(\ptr_V(f))$, for all $f\in\End(U\otimes V)$
\end{enumerate}
where $\ptr_V(f)$ is the right partial trace of $f$ obtained by using the duality morphism to close $f$ on the right. \\

c) Given an m-trace $\mt$ on $\ideal$, the {\em modified dimension} of $V\in \ideal$ is defined to be $\qd(V)=\mt_V(\Id_V)$.
\end{definition}
The notion of right and left ideals exist for pivotal
categories but they coincide in a ribbon category.

\begin{definition}
  
  Let $\catb$ be a 
  linear monoidal category. 

$1)$ A set of objects $ \mathcal D=\{ V_i \mid i \in J \} $ of $\catb$ is called a \emph{dominating set} if for any object $V \in \catb$ there exist indices $\{i_1,..,i_m \} \subseteq J $ and morphisms 
$\iota_k \in \Hom_{\catb}(V_{i_k},V), s_k \in \Hom_{\catb}(V,V_{i_k}),$ for all $ k\in \{1,...,m\}$ such that:
$$\Id_{V}=\sum_{k=1}^m \iota_k \circ s_k.$$ 

$2)$ A dominating set $\mathcal D$ is \emph{completely reduced} if:
$$\dim_\kk(\Hom_\catb(V_i,V_j))=\delta_{i,j}, \text{ for all } i,j \in J.$$
\end{definition}

\begin{definition}[Free realisation] \label{free}Let $\Zt$ be a commutative group with additive notation.  A \emph{free realisation} of $\Zt$ in $\mcC$ is a set of objects $\{\sigma(k)\}_{k\in \Zt}$ such that 
\begin{enumerate}
  \item $\sigma(0)=\unit$,
  \item the quantum dimension of $\sigma(k)$ is $\pm 1$ for all $k\in \Zt$,
  \item $\sigma(j)\otimes \sigma(k)= \sigma(j+k)$ for all $j,k\in \Zt$,
  \item $ \theta_{\sigma(k)}=\Id_{\sigma(k)}, \text{ for all } k \in \Zt $
where $\{\theta_V:V\to V\}_{V\in \mcC}$ is the twist in $\mcC$,
\item for any simple object $V$ in $\mcC$ we have $V\otimes \sigma(k) \cong V$ if and only if $k=0$.  
\end{enumerate}
\end{definition}

\begin{remark}
The above definition was first given in \cite{CGP14} with the condition that the quantum dimensions are all $1$.  This condition was relaxed in \cite{D17}.  Our definition of a free realisation is equivalent to the one given in \cite{D17}.  
\end{remark}

\begin{definition}[A $\Gr$-grading on a category] \label{D:Gstr}
Let $\Gr$ be a commutative group with additive notation.   A \textit{$\Gr$-grading} on $\mcC$ is an equivalence of linear categories
$\mcC \cong \bigoplus_{g \in \Gr} \mcC_g$ where 
$\{ \mcC_g \mid g \in \Gr \}$ is a  family of full subcategories of $\mcC$
satisfying the following conditions:
  \begin{enumerate}
  \item $\unit \in \mcC_0$,
  \item  if $V\in\mcC_g$,  then  $V^{*}\in\mcC_{-g}$,
  \item  if $V\in\mcC_g$, $V'\in\mcC_{g'}$ then $V\otimes
    V'\in\mcC_{g+g'}$,
  \item  if $V\in\mcC_g$, $V'\in\mcC_{g'}$ and $\Hom_\mcC(V,V')\neq 0$, then
    $g=g'$.  
    \end{enumerate}
\end{definition}

\begin{definition}\label{smallsymm}
Let $\Gr$ be a commutative group and a subset $\XX \subset \Gr$. We define the following two notions:
\begin{enumerate}
\item $\XX$ is \emph{symmetric} if $\XX=-\XX$.
\item $\XX$ is \emph{small} in $\Gr$ if for all $g_1,\ldots ,g_n\in \Gr$ we have:  
  $$ \bigcup_{i=1}^n (g_i+\XX) \neq \Gr.$$
\end{enumerate}

\end{definition}
 \begin{definition}(Relative pre-modular category)\label{D:1}
 Let $\Gr$ and $\Zt$ be commutative groups and $\XX$ be a small symmetric subset of $ \Gr$. Let $\mcC$ be a linear ribbon category over a field $\kk$ with the following  data:
 \begin{enumerate}
 \item A $\Gr$-grading on $\mcC$,
 \item A free realisation $\{\sigma(k)\}_{k\in \Zt}$ of $\Zt$ in $\mcC_0$,
 \item A non-zero m-trace $\mt$ on the ideal of projective objects of
   $\mcC$.
 \end{enumerate}
 A category $\mcC$ with this data is called a \emph{pre-modular $\Gr
$-category relative to $(\Zt,\XX)$} if the following properties are satisfied:
 \begin{enumerate}
\item \emph{Generic semisimplicity}: \label{D:5} for every $g \in \Gr \setminus \XX$, there exists a finite set of regular simple objects:
$$\Theta(g):=\{ V_i \mid i \in I_g  \}$$
such the following set is a completely reduced dominating set for $\mcC_g$:
$$\Theta(g) \otimes \sigma(\Zt):=\{ V_i \otimes \sigma(k) \mid i \in I_g, k\in \Zt  \}.$$

\item \emph{Compatibility}: \label{D:6}There exist a bilinear map $\psi: \Gr \times \Zt \rightarrow \kk^*$ such that:
  \begin{equation}
    \label{eq:psi}
    c_{\sigma(k),V}\circ c_{V,\sigma(k)}= \psi(g,k) \cdot  \Id_{V \otimes \sigma(k)}.
  \end{equation}
for any $g\in \Gr$, $V \in \mcC_g$ and $k \in \Zt$, where $c_{-,-}$ denotes the braiding on $\cat$.
\end{enumerate}
Throughout, we will denote the left and right duality morphisms on $\cat$ respectively by
\begin{equation*}
  \overrightarrow{\mathrm{ev}}_V:V^* \otimes V \to \mathds{1},
  \quad \overrightarrow{\mathrm{coev}}_V: \mathds{1} \to V \otimes V^*,
  \quad \overleftarrow{\mathrm{ev}}_V:V \otimes V^* \to \mathds{1},
  \quad \overleftarrow{\mathrm{coev}}_V:\mathds{1} \to V^* \otimes V \end{equation*}
 \end{definition}
In the following, we present the extra requirements that are needed in order to have a relative modular category.
\begin{definition}[Kirby color and Non-degenerate]\label{d:ndeg}
Let $\mcC$ be a pre-modular $\Gr$-category relative to $(\Zt,\XX)$.

\begin{enumerate}
  \item For $g \in \Gr \setminus \XX$, let the \emph{Kirby color of index $g$} be the following element: $$\Omega_g:= \sum_{i \in I_g}\qd(V_i) \cdot V_i.$$  \item For $g \in \Gr \setminus \XX$ and $V\in \mcC_g$ consider the scalars $\Delta_\pm\in \kk$ defined by
$$\epsh{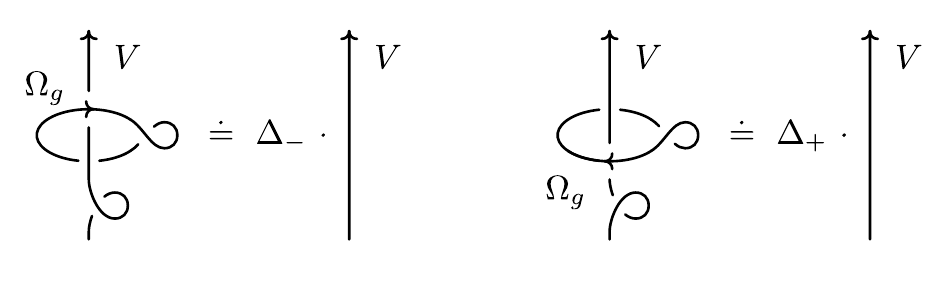}{20ex}
$$
where $ \dot{=}$ means that graphs are \emph{skein equivalent}, i.e.\ equal up to the Reshetikhin-Turaev functor associated with the category $\mcC$, see for example Section 1.2 of \cite{D17}.   Lemma 5.10 of \cite{CGP14} implies these scalars do not depend neither on $ V$ nor on $g$. We say $\mcC$ is \emph{non-degenerate} if $\Delta_{+}\Delta_{-}\neq 0$. 
\end{enumerate}
\end{definition}
We now state the additional condition required for a modular relative category.  
\begin{definition}[Relative modular category] \label{D:2}

We say that a category $\mcC$ is a modular $\Gr
$-category relative to $(\Zt,\XX)$ if:
\begin{enumerate}
\item $\mcC$ is a pre-modular $\Gr$-category relative to $(\Zt,\XX)$,
\item there exists a modularity parameter $\zeta_{\Omega} \in \kk^*$ such that for any $g,h \in \Gr \setminus \XX$ and any $i,j \in I_g$ one has:
  \begin{equation}\label{eq:mod}
    \epsh{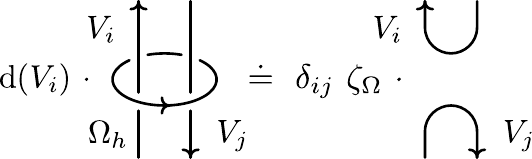}{14ex}\vspace*{4ex}.
  \end{equation}
\end{enumerate}
\end{definition}
As discussed above, non-degenerate pre-modular $\Gr$-categories give
rise to invariants of 3-manifolds with some additional structure.
Futhermore modular $\Gr$-categories give rise to TQFTs.  In \cite{D17}
it is shown that a modular $\Gr$-category is non-degenerate.

Through the article we will assume that if $\cat$ is a $\Gr$-graded
category, for any degree $g\in\Gr$, $\cat_g$ is not the null object
i.e. there exists $V\in\cat_g$ with $\Id_V\neq0$.

\subsection{Transparency}\label{trans}

Given a braided category $\cat$ with braiding $c_{-,-}$, an object
$V \in \cat$ is called transparent if
\[ c_{V,U} \circ c_{U,V} = \mathrm{Id}_{U \otimes V} \] for all
$U \in \cat$.  
\begin{definition}
  The Müger center of $\cat$ is the full subcategory of $\cat$
  consisting of all transparent objects.
\end{definition}

Similarly, given an endomorphism $f \in \mathrm{End}(V)$ for
$V \in \cat$, we call $f:V \to V$ transparent if
\begin{align*}
c_{V,U} \circ (f \otimes \mathrm{Id}_U) \circ c_{U,V}=\mathrm{Id}_U \otimes f\\
c_{W,V} \circ (\mathrm{Id}_W \otimes f) \circ c_{V,W}=f \otimes \mathrm{Id}_W\\
\end{align*}
for all $U,W \in \cat$. Clearly, $V$ is transparent if and only if $\mathrm{Id}_V$ is transparent. We have the following useful Lemma:

\begin{lemma}\label{subqtrans}
  Let $\cat$ be an abelian braided category with biexact tensor product and $V \in \cat$ a transparent object. Then all quotients and subobjects of $V$ are transparent.
\end{lemma}

\begin{proof}
  Let $U \in \cat$ be a subobject of $V$ with embedding
  $\iota: U \hookrightarrow V$. Then, since $V$ is transparent we have
  \[ c_{V,W} \circ c_{W,V}=\mathrm{Id}_{W \otimes V} \] for all
  $W \in \cat$. Multiplying both sides of this equation by
  $\mathrm{Id}_W \otimes \iota$ and applying naturality of the
  braiding, we see that
\begin{align*}
c_{V,W} \circ c_{W,V} \circ (\mathrm{Id}_W \otimes \iota) &= \mathrm{Id}_W \otimes \iota\\
\Longrightarrow \qquad  (\mathrm{Id}_W \otimes \iota) \circ c_{U,W} \circ c_{W,U}&=\mathrm{Id}_W \otimes \iota. 
\end{align*}
Since the tensor product is biexact $\iota:U \to V$ monic implies $\mathrm{Id}_W \otimes \iota$ is monic, so
\[ c_{U,W} \circ c_{W,U}=\mathrm{Id}_{W \otimes U}.\]
Therefore, $U$ is transparent. A nearly identical argument shows that quotients of transparent objects are transparent, so the lemma follows.
\end{proof}

Clearly the Müger center of $\cat$ contains the trivial objects of
$\cat$, i.e. objects which are a direct sum of copies of the unit
object.  We say the Müger center of $\cat$ is trivial if all its
objects are trivial.  Similarly,
we define for a relative pre-modular $\Gr$-category the $\Zt$-orbit of
$V$ to be the class of objects isomorphic to $V\otimes\sigma(k)$ for
some $k\in\Zt$.

\begin{definition}\label{Ztriv}
  Let $\cat$ be a relative pre-modular category $\mcC$ with
  translation group $\Zt$.  We say an object $X\in\cat$ is
  $\Zt$-trivial if it is a direct sum of objects in the $\Zt$-orbit
  of $\unit$.
  We say the Müger center of $\cat_0$ is $\Zt$-trivial if all its
  objects are $\Zt$-trivial.
\end{definition}
Since the braiding of
$\sigma(k)$ is governed by $\psi$ (recall Equation \eqref{eq:psi}), we have:
\begin{lemma}
  Suppose that $\cat_0$ has $\Zt$-trivial Müger center then $\cat$ has
  a trivial Müger center in degree $0$ if and only if the right kernel
  of $\psi$ is trivial in $\Zt$.  If furthermore the left kernel of
  $\psi$ is trivial in $\Gr$ then the Müger center of $\cat$ is
  trivial.
\end{lemma}
\begin{proof}
  If the right kernel of $\psi$ is not trivial, then any element
  $\sigma(k),\,k\in\Zt\setminus\{0\}$ such that $\psi(\cdot,k)=1$ is
  transparent but non trivial.  Reciprocally, if the Müger center of
  $\cat$ is non trivial in degree $0$ then there exists a non trivial
  simple transparent object $V$ of degree $0$.  Then $V$ is in the
  Müger center of $\cat_0$ and thus $V$ being $\Zt$-trivial is
  isomorphic to some $\sigma(k)\not\simeq\unit$.  Then for any
  $g\in\Gr$, the existence of a non zero $V\in\cat_g$ and $\sigma(k)$
  being transparent implies that $\psi(g,k)=1$.  So $k$ is in the
  right kernel of $\psi$.
  \\
  Next if the left kernel of $\psi$ is trivial in $\Gr$, then any
  object $V\in\cat_g$ of non trivial degree $g$, has non trivial
  braiding with any $\sigma(k)$ such that $\psi(g,k)\neq1$ thus $V$
  can't be transparent.  Hence the Müger center is concentrated in
  degree $0$ which conclude the proof. 
\end{proof}

\begin{definition}
  We say that the realization $(\sigma_k)_{k\in\Zt}$ has no self
  extension if for every Jordan-Holder series
  $M=M_0\supset M_1\supset\cdots\supset M_n=\{0\}$ in $\cat_0$ with
  $M_i/M_{i+1}$ in the $\Zt$-orbit of $\unit$, $M$ is $\Zt$-trivial
  (thus semisimple).
\end{definition}
This condition holds, for example, for all unrolled quantum groups by an easy argument on weights.

\subsection{$U^H_{\xi}(\mathfrak{sl}(m|n))$}\label{SS:slmn}
Let $A$ be the Cartan matrix of $\mathfrak{sl}(m|n)$ and $q$ an $\ell$-th root of unity. Then,
\begin{definition}
$U_q\mathfrak{sl}(m|m)$ is the $\mathbb{C}$-algebra with generators $E_i,F_i,K_i,K_i^{-1}$ with $i=1,...,m+n-1$ satisfying relations
\begin{align}
\label{A1}K_iK_j&=K_jK_i, &   K_iK_i^{-1}&=K_i^{-1}K_i=1,\\
\label{A2}K_iE_jK_i^{-1}&=q^{d_ia_{ij}}E_j, & K_iF_jK_i^{-1}&=q^{-d_ia_{ij}}F_j,\\
\label{A3}[E_i,F_j]&=\delta_{ij}\frac{K_i-K_i^{-1}}{q^{d_i}-q^{-d_i}} & E_m^2&=F_m^2=0\
\end{align}
and for $X_i=E_i,F_i$
\begin{align}
\label{A4}[X_i,X_j]=0 \quad &\mathrm{if} \; |i-j|>2\\
\label{A5}X_i^2X_j-(q+q^{-1})X_iX_jX_i+X_jX_i^2=0 \quad &\mathrm{if} \; |i-j|=1, \; \mathrm{and} \; i\not=m
\end{align}
\begin{align}
\label{A6}X_mX_{m-1}X_mX_{m+1}+&X_mX_{m+1}X_mX_{m-1}+X_{m-1}X_mX_{m+1}X_m\\
\nonumber &+X_{m+1}X_mX_{m-1}X_m-(q+q^{-1})X_mX_{m-1}X_{m+1}X_m=0
\end{align}
All generators are even except for $E_2,F_2$ which are odd and $[-,-]$ is the super-commutator given by $[x,y]=xy-(-1)^{|x||y|}yx$. $U_q(\mathfrak{sl}(m|n)$ can be given the structure of a Hopf superalgebra with coporoduct $\Delta$, counit $\epsilon$ and antipode $S$ by
\begin{align}
\label{H1}\Delta (E_i)&=E_i \otimes 1+K_i^{-1} \otimes E_i, & \epsilon(E_i)&=0, & S(E_i)&=-K_iE_i,\\
\label{H2} \Delta(F_i)&=F_i \otimes K_i+1 \otimes F_i, & \epsilon(F_i)&=0, & S(F_i)&=-F_iK_i^{-1}\\
\label{H3}\Delta(K_i^{ \pm 1})&=K_i^{\pm 1} \otimes K_i^{\pm 1} & \epsilon(K_i^{\pm 1})&=1 & S(K_i^{\pm 1})&=K_i^{\mp 1}
\end{align}
Then, the unrolled quantum group $U_q^H\mathfrak{sl}(m|n)$ is the $\mathbb{C}$-algebra with generators $H_i,E_i,F_i,K_i^{\pm 1}$ with $i=1,...,m+n-1$ and relations \eqref{A1}-\eqref{A6} plus the relations
\begin{equation}
\label{A7}[H_i,H_j]=[H_i,K_j^{\pm}]=0, \qquad [H_i,E_j=a_{ij}E_j, \qquad [H_i,F_j]=-a_{ij}F_j 
\end{equation}
$U_q^H\mathfrak{sl}(2|1)$ becomes a Hopf algebra with coproduct, counit, and antipode satisfying relations \eqref{H1}-\eqref{H3} and
\begin{equation}\label{H4} \Delta(H_i)=H_i \otimes 1 + 1 \otimes H_i, \qquad \epsilon(H_i)=0, \qquad S(H_i)=-H_i \end{equation}
\end{definition}
We call a $U_q^H\mathfrak{sl}(m|n)$-module $V$ a ``weight" module if
\begin{enumerate}
\item The elements $H_i,i=1,..,m+n-1$ act semisimple on $V$.
\item $K_i=q^{d_iH_i}$ as operators on $V$.
\end{enumerate}

We denote by $\mathcal{D}$ the category of $U_{q}^H\mathfrak{sl}(m|n)$ weight modules. Let $\mathfrak{h}:=\mathrm{Span}\{H_1,...,H_{m+n-1}\}$. Any weight module $V$ admits a basis $\{v_j\}_{j \in I}$ and a family of linear functionals $\{\lambda_j\}_{j \in I} \in \mathfrak{h}^*$ such that 
\[ H_iv_j=\lambda_j(H_i)v)j\]
for all $i$ and $j$. We call $\lambda_j$ the weight of $v_j$ and a weight of $V$ if it a weight for some $v \in V$. We call a weight $\lambda \in \mathfrak{h}^*$ of a module $V$ perturbative if $\lambda(H_i) \in \mathbb{Z}$ for all $i \not = m$ and a $U_{q}^H\mathfrak{sl}(m|n)$-module is called perturbative if all of its weights are. We denote the category of perturbative $U_{q}^H\mathfrak{sl}(m|n)$-modules by $\mathcal{D}^{\wp}$. Now, let $\mathsf{v}$ denote the standard $m+n$-dimensional module in $\mathcal{D}^{\wp}$ and let
$$S:=\{  V^{\lambda_z^0} \otimes \bar{\mathbb{C}}^{\otimes n} \otimes \varepsilon^m\; | \; n\in\{0,1\}, m \in \mathbb{Z}_{\geq 0}, \text{ and } \; z \in \mathbb{C} \text{ with } \bar{z}\in(\mathbb{C}/\mathbb{Z}) \setminus \XX^{\wp} \}$$
where $\XX^{\wp}=\left(\frac{1}{2}\mathbb{Z}\right)\setminus \mathbb{Z}$ and we have the following:
\begin{enumerate}
\item $\bar{\mathbb{C}}$ is the odd $1$-dimensional module with trivial action.
\item $\varepsilon$ is the $1$-dimensional module where $H_m$ acts by $\ell$, all $K_i$ act by $1$, and all other generators act by zero.
\item $V^{\lambda^0_z}$ is the unique simple highest weight module of highest weight $\lambda_z^0$ where
\[ \lambda_z^c= z\omega_m+\sum\limits_{\substack{k=1 \\ k \not = m}}^{r} c_k \omega_k. \]
for $z\in \mathbb{C},c_k \in\Z$, and $\omega_k$ are the fundamental weights.\\
\end{enumerate}
We denote by $\bar{\mathcal{D}}^{\wp}$ the full subcategory of $\mathcal{D}^{\wp}$ generated by $S$ with respect to tensor products, direct sums, and retracts.

\begin{definition}\label{DN}
We define $\mathcal{D}^{\aleph}$ to be the quotient of $\bar{\mathcal{D}}^{\wp}$ by its ideal of negligible morphisms.
\end{definition}
\section{Non-degeneracy}\label{SS:ND}

Let $\mcC$ be a pre-modular $\Gr$-category.  Let $g \in \Gr \setminus \XX$.  Recall the finite set of simple modules $\Theta(g)=\{V_i | i \in I_g\}$ in the definition of generic semisimplicity.  Consider the  \emph{$S'$-matrix} $S'=(S'_{ij})_{i,j\in I_g}$ where 
$S'_{ij}$ is the 
scalar corresponding to the endomorphism which is given by the partial trace of the double
braiding of $V_i$ and $V_j$:
\begin{equation}
  \label{eq:Sprime}
  \epsh{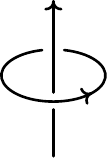}{12ex}\put(-8,-10){\ms{V_i}}\put(-40,24){\ms{V_j}}=S'_{ij}\,\Id_{V_j}.
\end{equation}
Let $S_g=(S_{ij})_{i,j\in I_g}$ be the \emph{modified $S$-matrix}
where $S_{ij}=\qd(V_j)S'_{ij}$ is the value of the modified RT-link
invariant evaluated on the Hopf link whose components are colored with
$V_i$ and $V_j$.  This is a symmetric matrix.  

\begin{lemma}\label{L:Snon-degS'}
  For fixed $g \in \Gr \setminus \XX$, the matrix $S_g$ is
  non-degenerate if and only if $S'_g$ is non-degenerate.
\end{lemma}
\begin{proof}
  The lemma follows from the fact that $S_g$ and $S'_g$ differ by the
  invertible diagonal matrix $D=(\qd(V_i))_{i\in I_g}$.
\end{proof}
In the rest of this section we will use two generalizations of $S$ and
$S'$: First for $g_1,g_2\in\Gr \setminus \XX$ we define similarly the
mixed $S$-matrix
$S_{g_1,g_2}=\bp{S_{ij}}_{(i,j)\in I_{g_1}\times I_{g_2}}$. It is a
$n_1\times n_2$ matrix where $n_i$ is the number of simple $\Zt$-orbits
in $\cat_{g_i}$.
\\
Next for any $U,V\in\cat$, let $\Phi_{U,V}=\ptr_U(c_{U,V}c_{V,U})$ so
that $\Phi_{V_i,V_j}=S'_{ij}\Id_{V_j}$ and
$S_{ij}=\mt\bp{\Phi_{V_i,V_j}}$.  This operator has the following
remarkable property:
$$\Phi_{U\otimes U',V}=\Phi_{U,V}\circ\Phi_{U',V}.$$
\begin{lemma}\label{L:rk-S-mixed}
  Suppose that the mixed $S$-matrices have no zero entry, then the
  rank of all (mixed) $S$-matrices is constant.
\end{lemma}
\begin{proof}
  If $g_1,g_2,g_3,g_4\in\Gr \setminus \XX$ with $g_3=g_1+g_2$, then
  for $(i_1,i_2,i_4)\in I_{g_1}\times I_{g_2}\times I_{g_4}$ we have a
  fusion formula
  $V_{i_1}\otimes V_{i_2}=\bigoplus_{i_3\in
    I_{g_3}}c_{i_1,i_2}^{i_3}V_{i_3}$.  
  Then since $V_{i_4}$ is simple we have
  $\mt\bp{\Phi_{V_{i_1}\otimes
      V_{i_2},V_{i_4}}}=\mt\bp{\Phi_{V_{i_1},V_{i_4}}\circ\Phi_{V_{i_2},V_{i_4}}}
  =\qd(V_{i_4})^{-1}\mt\bp{\Phi_{V_{i_1},V_{i_4}}}\mt\bp{\Phi_{V_{i_2},V_{i_4}}}$.
  This gives
  $$\qd(V_{i_4})^{-1}S_{i_1,i_4}S_{i_2,i_4}=\sum_{i_3\in
    I_{g_3}}c_{i_1,i_2}^{i_3}S_{i_3,i_4}.$$ Let us fix any $i_1$, then
  we have the equality of matrices:
  $S_{g_2,g_4}=\bp{c_{i_1,i_2}^{i_3}}_{i_2,i_3}\times
  S_{g_3,g_4}\times D$ where $D$ is the diagonal matrix
  $D_{i_4,i_4}=\qd(V_{i_4})S_{i_1,i_4}^{-1}$.  Thus
  $\rk(S_{g_2,g_4})\le\rk(S_{g_3,g_4})$.  Applying the same argument
  to $-g_1,g_3,g_2,g_4$ we get the inverse inequality thus we have
  $\rk(S_{g_2,g_4})=\rk(S_{g_3,g_4})$.  Finally the lemma follows
  since the matrix $S_{g_4,g_2}$ is the transpose of $S_{g_2,g_4}$.
\end{proof}
\begin{proposition}\label{DMug}
  Suppose that $\cat$ is abelian, $\{ \sigma(k)\}_{k\in\Zt}$ has no self
  extensions and that $\cat_0$ has $N$ simple
  $\Zt$-orbits. If for some $g\in\Gr\setminus\XX$, we have: 
\begin{enumerate}
\item $S_g$ is a $N\times N$ invertible matrix
\item Both $S_g$ and
  $S_{-g,g}$ have a row with no zero entry then the 
  Müger center of $\cat_0$ is $\Zt$-trivial. 
\end{enumerate}
then the Müger center of $\cat_0$ is $\Zt$-trivial.
\end{proposition}
\begin{proof}
  The hypothesis 2) implies that there exists $V_\alpha\in\cat_g$ and
  $V_\beta\in\cat_{-g}$ such that for any $j\in I_g$,
  $\Phi_{V_\alpha,V_j}\neq0$ and $\Phi_{V_\beta,V_j}\neq0$.

  Let us suppose that the Müger center of $\cat_0$ contains a non
  $\Zt$-trivial simple object $U$. 
  Let $\{U_j \; | \; j \in I_0 \}$ be a collection of representant of
  the simple $\Zt$-orbits of $\cat_0$ that contains $U_1= \unit$ and
  $U_t=U$.  Let $S_0^{\alpha}$ be the $N\times N$ matrix whose entries
  are
  $(S_0^{\alpha})_{ij}:=\mt\bp{\Phi_{V_{\alpha} \otimes U_i,
      V_{\alpha} \otimes U_j}}$. Applying the properties of braidings
  and the fact that $U_t$ is transparent in $\cat_0$, we see that
  \begin{align*}
    \Phi_{V_{\alpha} \otimes U_t,V_{\alpha} \otimes U_j}
    &=\Phi_{V_{\alpha} ,V_{\alpha} \otimes U_j}\circ\Phi_{U_t,V_{\alpha} \otimes U_j}\\
    &=\Phi_{V_{\alpha} ,V_{\alpha} \otimes U_j}\circ\bp{\Phi_{U_t,V_{\alpha}} \otimes \Id_{U_j}}\\
    &=\Phi_{V_{\alpha} ,V_{\alpha} \otimes U_j}\circ\bp{\qd(V_\alpha)^{-1}\mt\bp{\Phi_{U_t,V_{\alpha}}}
      \Id_{V_{\alpha}} \otimes \Id_{U_j}}\\
    &=\qd(V_\alpha)^{-1}\mt\bp{\Phi_{U_t,V_{\alpha}}}\Phi_{V_{\alpha} \otimes U_1,V_{\alpha} \otimes U_j}
\end{align*}
Taking the modified trace we get
$(S_0^{\alpha})_{tj}=\qd(V_\alpha)^{-1}\mt\bp{\Phi_{U_t,V_{\alpha}}}(S_0^{\alpha})_{1j}$
for all $j \in I_0$. Therefore, the $t$-th row of $S_0^{\alpha}$ is
proportional to the first row, so $S_0^{\alpha}$ is degenerate.


Let $M\in\cat_0$.  Then we claim that there exists $(T_i)_{i\in I_0}$
a sequence of $\Zt$-trivial modules such that for any projective
module $P\in\cat$, one has
$P\otimes M\simeq P\otimes \bp{\bigoplus_{i\in I_0}T_i\otimes U_i}$.
Indeed if $M=M_0\supset M_1\supset\cdots\supset M_n=\{0\}$ is a
Jordan-Hölder series of $M$ (with $M_j/M_{j+1}$ simple), then since $P$
is projective, the short exact sequence $0\to M_{j+1}\to M_j\to M_j/M_{j+1}\to 0$
splits when tensored by $P$ and
$P\otimes M_{j}\simeq \bp{P\otimes M_{j+1}}\oplus \bp{P\otimes
  M_j/M_{j+1}}$ where $M_j/M_{j+1}$ is isomorphic to
$\sigma(k_j)\otimes U_{i_j}$ for some $k_j\in\Zt$.  Then by induction
the claim follows with the modules
$T_i=\bigoplus_{U_{i_j}=U_i}\sigma(k_j)$.
We apply this to $P=V_\alpha$ and $M=V_\beta\otimes V_i\in\cat_0$
to produce $\Zt$-trivial objects $T_{i,j}$ such that
\begin{equation}
  \label{eq:VtoU}
  V_\alpha\otimes V_\beta\otimes V_i\simeq V_\alpha\otimes \bp{\bigoplus_{j\in
    I_0}T_{i,j}\otimes U_j}.
\end{equation}
On an other side, we can write a semisimple decomposition of
$V_\alpha\otimes U_k$: since the simple modules of $\cat_g$ are
isomorphic to $\sigma(z)\otimes V_i$ for some $z\in\Zt,i\in I_g$,
there exists $\Zt$-trivial modules $T'_{k,i}$ such that
\begin{equation}
  \label{eq:UtoV}
  V_\alpha\otimes U_k\simeq\bigoplus_{i\in I_g}T'_{k,i}\otimes V_i.
\end{equation}

Now for any $\Zt$-trivial object $T$ and any $g\in\Gr$, Equation
\eqref{eq:psi} implies that there exists a scalar $\Phi_{T|g}$ such
that for any $V\in\cat_g$, $$\Phi_{T,V}=\Phi_{T|g}\Id_V.$$ In
particular, we call $A,B,C,D$ the square $N\times N$ matrices
$A=\mt\bp{\Phi_{V_\alpha\otimes U_i,V_j}}_{i\in I_0, j\in I_g}$,
$B=(\Phi_{T_{i,j}|g})_{i\in I_g,j\in I_0}$,
$C=(\Phi_{T'_{i,j}|g})_{i\in I_0,j\in I_j}$ and $D$ the diagonal
matrix defined by $D_{ii}=\qd(V_i)^{-2}\mt\bp{\Phi_{V_\alpha,V_i}}\mt\bp{\Phi_{V_\beta,V_i}}$.
The hypothesis imply that $D$ is invertible.  Then from Equation
\eqref{eq:VtoU}, we have
\begin{align*}
  \Phi_{V_\alpha\otimes \bp{\bigoplus_{k\in I_0}T_{i,k}\otimes U_k},V_j}
  &=\sum_{k\in I_0}\Phi_{T_{i,k},V_j}\circ\Phi_{V_\alpha\otimes U_k,V_j}=\sum_{k\in I_0}B_{ik}\Phi_{V_\alpha\otimes U_k,V_j}\\
  =\Phi_{V_\alpha\otimes V_\beta\otimes V_i,V_j}
  &=\Phi_{V_\alpha,V_j}\circ\Phi_{V_\beta,V_j}\circ\Phi_{V_i,V_j}=D_{jj}S_{ij}\qd\bp{V_j}^{-1}\Id_{V_j}  
\end{align*}
and taking the modified trace, we get $BA=S_gD$. Hence $A$ is invertible.

Next from Equation \eqref{eq:UtoV},
$$A_{ij}=\mt\bp{\Phi_{V_\alpha\otimes U_i,V_j}}=\sum_k\mt\bp{\Phi_{T'_{i,k}\otimes V_k,V_j}}
=\sum_k\mt\bp{\Phi_{T'_{i,k},V_j}\circ\Phi_{V_k,V_j}}=\sum_kC_{ik}S_{kj}.$$
So $A=CS_g$ and $C$ is invertible.  Finally, 
\begin{align*}
  (S_0^{\alpha})_{ij}&=\mt\bp{\Phi_{V_{\alpha} \otimes U_i,V_{\alpha} \otimes U_j}}
                     =\mt\bp{\sum_k\Phi_{T'_{i,k}V_k,V_{\alpha} \otimes U_j}}\\
                     &=\sum_kC_{ik}\mt\bp{\Phi_{V_k,V_{\alpha} \otimes U_j}}
                     =\sum_kC_{ik}\mt\bp{\Phi_{V_{\alpha} \otimes U_j,V_k}}\\
                     &=\sum_kC_{ik}A_{jk}
\end{align*}
So $S_0^{\alpha}=C\,{^t\!A}$ is invertible and we have a contradiction.

Hence all simple objects of the Müger center of $\cat_0$ are in the
orbit of $\unit$. Now for any transparent object $M$ in $\cat_0$ all
the subobjects in a Jordan-Holder series of $M$ are transparent by
Lemma \ref{subqtrans}, and so are their simple quotient.  Hence all
these simple are in the $\Zt$-orbit of $\unit$ and since
$(\sigma_k)_{k\in\Zt}$ has no self extension, $M$ is $\Zt$-trivial.
\end{proof}

\begin{theorem}\label{T:Pre-modNonDeg}
  A pre-modular $\Gr$-category $\mcC$ is non-degenerate if there
  exists $g \in \Gr \setminus \XX$ such that the corresponding
  $S$-matrices $S_g$ and $S_{-g,g}$ are non-degenerate.
\end{theorem}

\begin{proof}
Fix such a $g \in \Gr \setminus \XX$.  We need to prove that $\Delta_{+}\Delta_{-}\neq 0$.  For $i\in I_g$, let $t_i$ be the scalar determined by the twist on $V_i$.  If $V=V_j$ in the Definition \ref{d:ndeg} then the scalar $\Delta_{-}$ is equal to 
\begin{equation}\label{E:D+}
   \Delta_{-}= t_j^{-1} \sum_{i\in I_g} S'_{ij} t_i^{-1} \qd(V_i)
\end{equation}
and it does not depend of $g\in \Gr \setminus \XX$ nor of $j\in I_g$.
Now let $T=(\delta^i_{j}t_{j}^{-1})_{i,j\in I_g}$ be the invertible diagonal matrix of twists inverse.  Also, let $T_c=(t_i^{-1} \qd(V_i))_{i\in I_g}$ be the column vector of inverse twists times modified dimensions.    If $ \Delta_{-}$ is zero then Equation \eqref{E:D+} can be written as the following matrix multiplication:  
$$T(S'_{ij})T_c=\bar 0.$$
This would imply that the matrix $(S'_{ij})$ would be degenerate.  But
by Lemma \ref{L:Snon-degS'} this would be a contradiction.  Thus,
$ \Delta_{-}$ is non-zero.  Similarly, using that $(S_{i^*j})$ is non
degenerate, $\Delta_{+}$ is non-zero and the theorem follows.
\end{proof}

\section{Modularity / Transparent Morphisms}\label{relmod}

\begin{lemma}\label{trans}
Let $\mcC$ be an abelian braided category with left or right exact tensor product. If $V \in \mcC$ and $f:V \to V$ is transparent, then $\mathrm{Im}(f)$ is a transparent object.
\end{lemma}

\begin{proof}
We assume first the tensor product is left exact. Since $f$ is transparent, we have
\begin{align*}
c_{V,U} \circ (f \otimes \mathrm{Id}_U) \circ c_{U,V}&=\mathrm{Id}_U \otimes f,\\
c_{W,V} \circ (\mathrm{Id}_W \otimes f) \circ c_{V,W}&=f \otimes \mathrm{Id}_W,
\end{align*}
for any $U,W \in \mcC$. $\mcC$ is abelian, so we can decompose $f:V \to V$ as $f=\iota \circ f|$ where $f|:V \to \mathrm{Im}(f)$ is epic and $\iota:\mathrm{Im}(f) \to V$ is monic. Substituting this formula and applying naturality of the braiding isomorphism to the first equation gives
\begin{align*}
(\mathrm{Id}_U \otimes \iota ) \circ( \mathrm{Id}_U \otimes f|)&=c_{V,U} \circ (\iota \otimes \mathrm{Id}_U) \circ (f| \otimes  \mathrm{Id}_U ) \circ c_{U,V} \\
&= (\mathrm{Id}_U \otimes \iota) \circ c_{\mathrm{Im}(f),U} \circ c_{U,\mathrm{Im}(f)} \circ (\mathrm{Id}_U \otimes f|)
\end{align*}
for all $U \in \mcC$. Since the tensor product is left exact and $\iota$ is monic and $f|$ is epic, it follows that $\mathrm{Id}_U \otimes \iota$ is monic and $\mathrm{Id}_U \otimes f|$ is epic, therefore
\[ c_{\mathrm{Im}(f),U} \circ c_{U,\mathrm{Im}(f)}=\mathrm{Id}_{U \otimes \mathrm{Im}(f)}\]
for all $U \in \mcC$, so $\mathrm{Im}(f)$ is transparent. If the tensor product is right exact, start with the second equation rather than the first.
\end{proof}

\begin{lemma}\label{factor}
Let $\mcC$ be a graded category satisfying the conditions of Lemma \ref{trans}. Then, if $f:V \to V$ is transparent in $\mathcal{C}_0$, it factors through the M\"{u}ger center of $\mathcal{C}_0$. Further, if $\mathcal{C}$ admits a free realization of an abelian group $\Zt$ in $\mathcal{C}_0$ such that the Müger center of $\mcC_0$ is $Z$-trivial (recall Definition \ref{Ztriv}), then 
\begin{enumerate}
\item[1)] If $f:V \to V$ is transparent in $\mcC_0$, then $f$ factors through the free realization of $\mcC$. That is, there exists some integer $m$ and morphisms $g_i \in \Hom_{\mcC}(V,\sigma(k_i))$ and $h_i \in \Hom_{\mcC}(\sigma(k_i),V)$ for $i \in \{1,...,m\}$ such that
\[ f=\sum\limits_{i=1}^m h_i \circ g_i.\]
\item[2)] If $f:V \otimes V^* \to V \otimes V^*$ is a transparent endomorphism in $\mcC_0$ for some simple object $V$, then $f=\lambda (\overrightarrow{\mathrm{coev}}_V \circ \overleftarrow{\mathrm{ev}}_V)$ for some scalar $\lambda \in \mathbb{C}$.
\end{enumerate}
\end{lemma}

\begin{proof}
Since $\mcC$ is abelian, we can decompose $f:V \to V$ as $f=\iota \circ f|$ where $f|:V \to \mathrm{Im}(f)$ is surjective and $\iota:\mathrm{Im}(f) \hookrightarrow V$ is injective. Since $f$ is transparent, it follows from Lemma \ref{trans} that $\mathrm{Im}(f)$ is transparent and therefore lies in the M\"{u}ger center of $\mathcal{C}_0$, so $f$ factors through the M\"{u}ger center. If the M\"{u}ger center of $\mathcal{C}_0$ is $Z$-trivial, then
\[ \mathrm{Im}(f) \cong \bigoplus\limits_{k \in Z} c_k^f \sigma(k)\]
for some non-negative integers $c_k^f$. It then follows that 
\[ f|=\sum\limits_{i=1}^m g_i, \qquad \iota=\sum\limits_{i=1}^m h_i\]
where $g_i \in \Hom_{\mcC}(V,\sigma(k_i))$, $h_i \in \Hom_{\mcC}(\sigma(k_i),V)$, and $m=\sum\limits_{k \in Z} c_k^f$, proving the first statement. For any $V \in \mcC$, we have 
\begin{equation}\label{eqhom}
 \mathrm{Hom}(\sigma(k),V \otimes V^*) \cong \mathrm{Hom}(\sigma(k) \otimes V,V) \cong \mathrm{Hom}(V \otimes V^*, \sigma(-k))
\end{equation}
for all $k \in Z$, where we have used the fact that the dual of $\sigma(k)$ is $\sigma(-k)$. Recall from Definition \ref{free} that for any simple module $V$, $V \otimes \sigma(k) \cong V$ if and only if $k=0$. Since we now assume $V$ is simple, it therefore follows from equation \eqref{eqhom} that $\mathrm{Im}(f) \cong \mathds{1} \subset V \otimes V^*$, so $f$ factors through the unit object
\[ f:V \otimes V^* \overset{f|}{\twoheadrightarrow} \mathds{1} \overset{\iota}{\hookrightarrow} V \otimes V^* \]
It also follows from equation \eqref{eqhom} with $k=0$ that $\mathrm{Hom}(V \otimes V^*,\mathds{1})$ and $\mathrm{Hom}(\mathds{1},V \otimes V^*)$ are both one-dimensional, so we have
\[ f=\lambda(\overrightarrow{\mathrm{coev}}_V \circ \overleftarrow{\mathrm{ev}}_V)\]
for some $\lambda \in \mathbb{C}$.
\end{proof}

For $g,h\in \Gr \setminus \XX$ and $i,j\in I_g$, let
$\Phi_{ij,g}^h:V_i \otimes V_j^* \to V_i \otimes V_j^*$ denote the
morphism represented in the left hand side of the relative modular
condition given in Equation \eqref{eq:mod}:
$$\Phi_{ij,g}^h=\qd(V_i)\Phi_{\Omega_h,V_i\otimes V_j^*}=\qd(V_i)\epsh{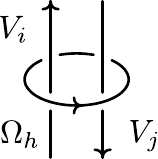}{8ex}$$\\

\begin{theorem}\label{degenmod}
  Let $\mcC$ be a relative pre-modular category.
\begin{enumerate}
\item If $\mcC$ is relative modular, then $S_g$ is non-degenerate for some generic index $g \in \Gr \setminus \XX$.
\item If $\mcC$ is unimodular and $S_g$ is non-degenerate for some generic index $g \in \Gr \setminus \XX$ and $\cat_0$ has $\Zt$-trivial Müger center, then $\mcC$ is a relative modular category.
\end{enumerate}
\end{theorem}

\begin{proof} First, suppose $\mathcal{C}$ is relative modular. By taking the $m$-trace of both sides of Equation \eqref{eq:mod} with $V_i=V_j$ and $\Omega_h=\sum\limits_{k \in I_h} d(V_k)V_k$, we find that
  \[ \sum\limits_k d(V_k)d(V_i)S'_{k,i}S'_{i^*,k}=\zeta_{\Omega} \not
    = 0. \] This means that the product of matrices
  $\frac1\zeta_{\Omega}S_{g,h}S_{h,-g}$ is the identity matrix.  In
  particular for $g=h$, these square matrices $S_g$ and $S_{g,-g}$ are
  non degenerate.


  We now prove the second statement.  Let $g,h\in \Gr \setminus
  \XX$. We proceed by proving the following three statements:
\begin{enumerate}
\item The morphism $\Phi_{ij,g}^h$ is transparent in $\mcC_0$ for all $i,j\in I_g$. 
\item If $S_{g}$ is non-degenerate for some $g \in \Gr \setminus \XX$, then $\Phi_{ij,g}^h \not = 0$ if and only if $i=j$. Moreover, for $i\in I_g$ there exists $\zeta^h_{i,g}\in \kk$ such that $\Phi_{ii,g}^h=\zeta^h_{i,g}(\overleftarrow{\mathrm{coev}}_{V_i} \circ \overrightarrow{\mathrm{ev}}_{V_i})$.
\item The scalar  $\zeta^h_{i,g}\in \kk$ equals $\Delta_+\Delta_-$ and does not depend on $g,h \in \Gr \setminus \XX$ nor $i \in I_g$. 
\end{enumerate}
It is then clear that $\cat$ satisfies Definition \ref{D:2}.\\

1. Let $V_i,V_j \in \mcC_g$ so $\Phi_{ij,g}^h:V_i \otimes V_j^* \to V_i \otimes V_j^*$ is an edomorphism in $\mcC_0$. The proof of Lemma 5.9 of \cite{CGP14} shows that in any pre-relative modular category the handle slide property holds (note this lemma is proved in the context of non-degenerate pre-relative modular categories but the non-degeneracy condition is not used in the proof).  Let $W$ be an object in $\mcC_0$ and consider the morphism $ \Id_W \otimes \Phi_{ij,g}^h$ and corresponding diagram which is just the diagram on the left side of Equation \eqref{eq:mod} with a disjoint string labeled with $W$ on the left.   The handle slide property says we can slide the component labeled with $W$ over the component labeled with the Kirby color.  Doing this we obtain the equality
$$
 \Id_W \otimes \Phi_{ij,g}^h = c_{V_i\otimes V_j,W} \circ  ( \Phi_{ij,g}^h \otimes  \Id_W )  \circ  c_{W, V_i\otimes V_j}.
$$
An analogous argument shows $\Phi_{ij,g}^h \otimes  \Id_W=  c_{W, V_i\otimes V_j} \circ(\Id_W \otimes \Phi_{ij,g}^h ) \circ c_{V_i\otimes V_j,W}$.   Thus, $ \Phi_{ij,g}^h$ is transparent in $\mcC_0$.  \\

2. Since the M\"{u}ger center of $\mathcal{C}_0$ is $Z$-trivial, it follows from transparency of $\Phi_{ij,g}^h$ and Lemma \ref{factor} that 
\[ \Phi_{ij,g}^h=\sum\limits_{s=1}^n h^{ij}_s \circ g^{ij}_s\]
for some $h_s^{ij} \in \Hom_{\mcC}(V_i \otimes V_j^*,\sigma(k_s))$ and $g_s^{ij} \in \Hom_{\mcC}(\sigma(k_s),V_i \otimes V_j^*)$. Recall from Definition \ref{free} that $\mathrm{Hom}(V_i,\sigma(k) \otimes V_j)=0$ unless $i=j$ and $k=0$, so we must have $\Phi_{ij,g}^h=0$ if $i \not = j$. Now, let $X$ be the non-zero matrix defined by
\[ X^T=(S'_{1,i^*},...,S'_{N,i^*}) \]
Then, by non-degeneracy of $S_g$, $S_gX$ is a non-zero vector. However, $(S_gX)_j=\sum\limits_{k \in I_g} \qd(V_k)S'_{j,k}S'_{k,i^*}=\Phi_{ji,g}^h$ and $\Phi_{ji,g}^h=0$ for all $j \not =i$. Therefore, $\Phi_{ii,g}^h \not= 0$. Finally, since the Müger center of $\cat_0$ is trivial, Lemma \ref{factor} together with (1) implies the existence of $\zeta^h_{i,g}\in \kk$.

3. We have $V_i \in \cat_g$ is generic, so
$V_i \otimes V_i^*\cong \oplus_{k\in K} P_k$ where $K$ is an indexing
set and $P_k$ are indecomposable projective objects. Since the
evaluation and coevaluation give non-zero maps between the unit object
$\unit$ and $V_i \otimes V_i^*$ we have that this direct sum contains
the projective cover $P_0$ of the unit object $\unit$.  For $k\in K$
let $S'_{P_k}:P_k\to P_k$ be the morphism corresponding to the diagram
\begin{equation*}
  \epsh{Sprime}{12ex}\put(-8,-10){\ms{\Omega_h}}\put(-40,24){\ms{P_k}}.
\end{equation*} 
It follows again from \cite[Lemma 5.9]{CGP14} that $S'_{P_k}$ is
transparent for all $k \in K$. Since $S'_{P_k}$ is transparent it
factors through the translation group by Lemma \ref{factor}. However,
by Equation \eqref{eqhom} and Definition \ref{free} (part 5), we have
$\dim \mathrm{Hom}(\sigma(k), V_i^* \otimes V_i^*)=\delta_{k,0}$ and
we know that $\dim \mathrm{Hom}(\sigma(0),P_0) = 1$ since $\cat$ is unimodular. Therefore, $S'_{P_k}=0$ for all $k \not = 0$. To compute
the scalar $\zeta^h_{i,g}$ we use the m-trace:
\begin{align*}
\label{}
  \qd(V_i)\zeta^h_{i,g}&= \zeta^h_{i,g}\mt_{V_i\otimes V_i^*}(\overleftarrow{\mathrm{coev}}_{V_i} \circ \overrightarrow{\mathrm{ev}}_{V_i})  \\
    &  =\mt_{V_i\otimes V_i^*}(\Phi_{ii,g}^h)\\
    &=\sum_{k\in K}\mt_{V_i\otimes V_i^*}\left((\imath_k\circ p_k)\circ\Phi_{ii,g}\right)\\
&= \sum_{k\in K}\mt_{V_i\otimes V_i^*}\left(\imath_k \circ  \qd(V_i) S'_{P_k}\circ p_k \right)\\
&= \qd(V_i)\mt_{V_i\otimes V_i^*}\left(\imath_0 \circ S'_{P_0}\circ p_0 \right)\\
&= \qd(V_i)\mt_{P_0}\left(S'_{P_0} \right)
\end{align*}
where $p_k: V_i\otimes V_i^*\to P_k$ and $\imath_k: P_k\to V_i\otimes V_i^*$ are projections and injections representing the direct sum.  Thus, $\zeta^h_{i,g}=\mt_{P_0}\left(S'_{P_0} \right)$ and does not depend on $g,h$ nor $i$. Let $\zeta=\mt_{P_0}\left(S'_{P_0} \right)$.  Then we can compute
in two ways the left partial trace of $\sum_i\theta_i\Phi_{ij,g}^h$:
on one side it gives
$\ptr_L(\sum_i\theta_i\Phi_{ij,g}^h)=\zeta\theta_j\Id_{V_j}$ and on
the other side,
$\ptr_L(\sum_i\theta_i\Phi_{ij,g}^h)=\Delta_+\Delta_-\theta_j\Id_{V_j}$
so $\zeta=\Delta_+\Delta_-$.  
\end{proof}

This theorem generalizes \cite[Proposition 3.36]{AGP} and of course applies to the category $\mathcal{D}$ of weight modules for $U_{\xi}^H\mathfrak{sl}(m|n)$ at odd root of unity $\xi$. The only condition of the Theorem which isn't explicitly shown in \cite{AGP} is $Z$-triviality of the M\"{u}ger center of $\mathcal{C}_0$, which appears implicitly in the proof of the factorization property \cite[Lemma 3.35]{AGP}. This property is not difficult to show directly, and follows from the fact that lying in the M\"{u}ger center places strong restrictions on which weights spaces can be non-empty. We expect that Theorem \ref{degenmod} should apply to the other unrolled super quantum groups as well. We also have the following corollary of the proof of the previous theorem:
\begin{corollary}\label{C:DeltaDelta}
If $\mathcal{C}$ is a relative modular category then $\zeta_{\Omega}\ = \Delta_+\Delta_-$ (recall Definition \ref{d:ndeg} \& \ref{D:2}).
\end{corollary}

\section{Generic Semisimplicity}\label{gens}

In this section we consider categorical constructions mirroring those used in \cite{AGP} to construct a pre-relative modular category from the category of perturbative modules for $U_{\xi}^H\mathfrak{sl}(m|n)$ at root of unity $\xi$. We derive properties of these categories which will be useful for proving relative modularity in various examples. First, we need to recall the notion of negligibility for objects.

\begin{definition}
Let $\mathcal{C}$ be a category admitting a modified trace $\{\mt_V\}_{V \in \mathcal{C}}$ (recall Definition \ref{mtrace}). We call an object $V \in \mathcal{C}$ negligible if $\mt_V(f)=0$ for all $f \in \mathrm{End}(V)$.
\end{definition}

Negligible objects form an ideal \cite[Lemma 5.6]{AGP} which we denote by $\mathcal{N}(\mathcal{C})$.

\begin{definition}\label{D:WS}
An object $X \in \mcC$ is called weak semisimple if its image $\overline{X} \in \mcC/\mcN(\mcC)$ is semisimple. A weak semisimple object has strong decomposition if 
\[ X \cong S \oplus N\]
where $S$ is semisimple and non-negligible while $N$ is negligible.
\end{definition}

Suppose $\tilde{X} \cong S \oplus N$ has strong decomposition and $X \cong \bigoplus_{k} X_k$ is a retract of $\tilde{X}$ with $X_k$ its indecomposable factors. Then we have morphisms 
\[r:\tilde{X} \to X \qquad \mathrm{and} \qquad t:X \to \tilde{X}\]
 such that $r \circ t = \mathrm{Id}_X$. If $t_k:=t|_{X_k}$, we have $r \circ t_k=\mathrm{Id}_{X_k}$ so $X_k$ is a retract of $\tilde{X}$. Since $X_k$ is indecomposable, we have $\mathrm{Im}(t_k) \subset S \text{ or } N$, so by restricting $r$ to $S$ or $N$ as appropriate, we see that each indecomposable factor $X_k$ is a retract of either $S$ or $N$. Since $S$ is semisimple, any indecomposable retract of $S$ will be simple, and any retract of $N$ is negligible by \cite[Lemma 5.6]{AGP}. It is also clear that direct sums of objects with strong decompositions have strong decomposition. We therefore have the following lemma:

\begin{lemma}\label{ret}
Direct sums and retracts of objects with strong decomposition have strong decomposition.
\end{lemma}

\begin{definition}\label{D:CS}
Let $\mcC$ be linear ribbon category. Given any subset $S \subset \mathrm{Ob}(\mcC)$, let $\mcC^S$ be the full subcategory of $\mcC$ generated by objects in $S \cup S^*$ with respect to tensor products, direct sums, and retracts where $S^*:=\{ X^* \; | \; X \in S \}$.
\end{definition}

Recall that if a $\Gr$-graded category $\mcC$ has a small symmetric subset $\XX \subset \Gr$ (recall Definition \ref{smallsymm}), then we call the objects in $\mcC_g$ with $g \in \Gr \setminus \XX$ generic. For any $S \subset \mathrm{Ob}(\mcC)$, $\mcC^S$ is a full subcategory of $\mcC$, so $\mcC^S$ inherits the grading of $\mcC$ in the obvious way.

\begin{theorem}\label{sd}
Let $\mcC$ be a linear ribbon category and $S \subset \mathrm{Ob}(\mcC)$. Then,
\begin{enumerate}
\item[1)] If all objects in $S$ and all tensor products of objects in $S$ have strong decomposition, then all objects in $\cat^S$ have strong decomposition.
\item[2)] Suppose $\mcC$ is graded by a group $G$ with small symmetric subset $\XX$. If all generic objects in $S$ and all generic tensor products of objects in $S$ have strong decomposition, then all generic objects in $\mcC^S$ have strong decomposition.
\end{enumerate}
\end{theorem}
\begin{proof}
To prove this statement, we introduce a filtration on $\mcC$. Let $\mcC_0=S$ and for $k \geq 1$, let $\mcC_k$ be the full subcategory of $\mcC$ generated by tensor products of objects in $\mcC_{k-1}$ with respect to direct sums and retracts. Clearly then, we have
\[ \mcC_0 \subset \mcC_1 \subset \cdots \subset \mcC_k \subset \cdots  \qquad \text{with} \qquad \mcC=\bigcup\limits_{k=1}^{\infty} \mcC_k. \]

We prove by induction on $k$ that all indecomposable objects in $\mathcal{C}$ are either in $S$, a tensor product of objects in $S$, or a retract of such objects. It will then follow from Lemma \ref{ret} and the stated assumptions that all objects in $\mathcal{C}$ have strong decomposition. The statement is clearly true for $\mathcal{C}_0$, so assume now that it is true for $\mathcal{C}_{k-1}$ and let $X \in \mathcal{C}_k \setminus \mathcal{C}_{k-1}$ be indecomposable. By construction, $X$ is either a tensor product of objects in $\mathcal{C}_{k-1}$ or a retract of such a tensor product. If $X= \bigotimes\limits_{k=1}^n Y_k$ with $Y_k \in \mathcal{C}_{k-1}$, then by the induction assumption $Y_k$ is in $S$, is a tensor product of objects in $S$, or is a retract of such objects. If $Y_k$ is in $S$ or is a tensor product of objects in $S$, let $\tilde{Y}_k=Y_k$ and if $Y_k$ is a retract of such objects, let $\tilde{Y}_k$ be the tensor product of objects in $S$ such that $Y_k$ is a retract of $\tilde{Y}_k$. Clearly then, $X$ is a retract of $Y:=\bigotimes\limits_{k=1}^n \tilde{Y}_k$ and $Y$ is a tensor product of objects in $S$. Any retract of $X$ will also be a retract of $Y$, so the induction argument is complete. It follows now from the assumption that all tensor products of objects in $S$ have strong decomposition and Lemma \ref{ret} that all object in $\mathcal{C}$ have strong decomposition.

The argument for part $(2)$ is almost identical to $(1)$. One only needs to keep in mind that the retract of an object in $\mcC_g$ is also in $\mcC_g$. That is, $X$ is generic if and only if any retract of $X$ is generic.
\end{proof}

It then follows trivially that if $S$ satisfies the consitions of Theorem \ref{sd} $(1)$, then $\overline{\mcC}^S$ is semisimple, and if $S$ satisfies the conditions of Theorem \ref{sd} $(2)$, then $\overline{\mcC}^S$ is generically semisimple. It turns out that many interesting examples have a distinguished object which makes checking the conditions of Theorem \ref{sd} easier, as described by the following corollary of Theorem \ref{sd}:

\begin{corollary}\label{Cor1}
Let $\mcC$ be a linear ribbon category and let $S^{\mathsf{v}}:=S \cup \{\mathsf{v}\}$ where $S \subset \mathrm{Ob}(\cat)$ consists of objects with strong decomposition and $\mathsf{v} \in \mathrm{Ob}(\cat)$ is a distinguished object with strong decomposition such that:
\begin{enumerate}
\item If $V \in S$ then $V \otimes \mathsf{v}^n$ has strong decomposition for all $n \in \mathbb{Z}_{\geq 0}$.
\item For any $U_1,U_2 \in S$, $U_1 \otimes U_2 \cong \bigoplus \limits_k V_k$ where $V_k$ is a retract of $\tilde{V}_k \otimes \mathsf{v}^{n_k}$ for some $\tilde{V}_k \in S$ and $n_k \in \mathbb{Z}_{\geq 0}$.
\end{enumerate}
Then, every object in $\cat^{S^{\mathsf{v}}}$ has strong decomposition.\\
\end{corollary}

The proof of this corollary is an easy induction to show that any tensor product of objects in $S^{\mathsf{v}}$ is a direct sum of retracts of objects of the form $V \otimes \mathsf{v}^n$ for some $V \in S$ by property (2). One then applies condition (1) and Theorem \ref{sd}. An identical argument adding generic conditions where necessary provides the following:
\begin{corollary}\label{Cor2}
Let $\mcC$ be a linear ribbon category graded by a group $\Gr$ with a small symmetric subset $\XX$.  As above, let $S^{\mathsf{v}}:=S \cup \{\mathsf{v}\}$ where $S \subset \mathrm{Ob}(\cat)$ consists of objects with strong decomposition and $\mathsf{v} \in \mathrm{Ob}(\cat)$ is a distinguished object with strong decomposition.  Suppose  that the following three conditions hold, where a generic object is any object in $\mathcal{C}_g$ with $g \in \Gr \setminus \XX$:
\begin{enumerate}
\item Every generic object in $S^{\mathsf{v}}$ has strong decomposition.
\item Given any $V \in S$ and $n \in \mathbb{Z}_{\geq 0 }$, if $V \otimes \mathsf{v}^n$ is generic, then it has strong decomposition.
\item For $U_1 ,U_2 \in S$, if $U_1 \otimes U_2$ is generic, then $U_1 \otimes U_2 \cong \bigoplus\limits_k V_k$ where $V_k$ is a retract of $\tilde{V}_k \otimes \mathsf{v}^{n_k}$ for some $\tilde{V}_k \in S$ and $n_k \in \mathbb{Z}_{\geq 0}$.
\end{enumerate}
Then, every generic object in $\mathcal{C}^{S^{\mathsf{v}}}$ has strong decomposition.\\
\end{corollary}
It is easy to see that the category $\bar{\mathcal{D}}^{\wp}$ of Subsection \ref{SS:slmn} can be realized as $\mathcal{C}^{S^{\mathsf{v}}}$ for some choice of $S$ and $\mathsf{v}$. Condition $(1)$ of Corollary \ref{Cor2} holds by construction and a mild adjustment to the argument of \cite[Lemma 4.11]{AGP} (to allow for odd irreducible modules) together with \cite[Fact 5.8.2]{AGP} proves that condition $(3)$ holds. Finally, condition $(2)$ follows from \cite[Fact 5.8.1]{AGP}. We therefore see that the arguments in \cite{AGP} required to show that all generic objects in $\bar{\mathcal{D}}^{\wp}$ have strong decomposition can now be compressed into a few simple lemmas. We expect that these results will be useful for showing that other super quantum groups have similar properties. It is clear then that if a category $\mathcal{C}$ satisfies the conditions of Corollary \ref{Cor1}, then the quotient $\overline{\mathcal{C}}:=\mathcal{C}/\mathcal{N}(\mathcal{C})$ of $\mathcal{C}$ by its ideal $\mathcal{N}$ of negligible morphisms is generically semisimple.

\section{$U_{\xi}^H\mathfrak{sl}(2|1)$}\label{SS:U}

In \cite{AGP} it was conjectured that the category $\mathcal{D}^{\aleph}$ of Definition \ref{DN} for $U_{\xi}^H\mathfrak{sl}(2|1)$ is relative modular.  We now show that the $S$-matrix for this category is degenerate which disproves this conjecture by Theorem \ref{degenmod}.  Interestingly, while disproving the conjecture of \cite{AGP} we see how to possibly modify the translation group to get a non-degenerate $S$-matrix and a new conjecture postulating that the category with this translation group is relative modular category.  

In \cite{AGP}, also see Subsection \ref{SS:slmn}, for $U_{\xi}^H\mathfrak{sl}(2|1)$  the grading on $\mathcal{D}^{\aleph}$ is given by $G=\C/\Z$ and the small symmetric set is   $\XX=\{\bar0\}\subset G$.  Also, as in \cite{AGP}, let $\Zt=\Z/2 \times \Z$ and let the translation
group be $\{\sigma(\bar z, k)\}_{(\bar z,k) \in\Zt}$ where $\sigma(\bar z, k)$ is the 1-dimensional module generated by a vector of super degree $\bar z$ where $K_1, K_2$ act by 1, $H_2$ acts by $k\ell$ and all other generators act by zero.   
Finally, for $\alpha\in\C\setminus\Z$, as in Proposition 5.14 of \cite{AGP} the finite set of regular simple objects 
$$\Theta_{\wb \alpha}=\{V(\lambda^k_{\alpha+i}): 0\le k\le \ell-2,0\le i\le\ell-1\}
$$
determine a completely reduced dominating set for $\mathcal{D}^{\aleph}_{\wb \alpha}$ and $\mathcal{D}^{\aleph}$ is a Relative pre-modular category, as in Definition \ref{D:1}.  
With these choices the $S$-matrix is an $\ell(\ell-1) \times \ell(\ell-1)$ matrix.  
%

To show this S-matrix is non-degenerate we consider the following module.  For an integer $k$ such that $1\leq k\leq \ell -1$, the Lie superalgebra $\mathfrak{sl}(2|1)$ has a $2k+1$-dimensional
simple module obtained as the super-symmetric power
$S^k(\C^{(2|1)})$ with highest weight $k\omega_1$, where $\omega_1$ is the first fundamental weight. 
 This module has a deformation $A_k$ with the same (super)-character which remains simple when $q$ is a root of unity.  We give here a
complete description of $A_k$.

\begin{lemma}
  Let $v_0^0$ be a highest weight vector of $A_k$.  Then $A_k$ has a
  basis with vectors $\big\{v_i^j\big\}_{0\le j\le1,\,0\le i\le k-j}$ of
  parity $j$ and for which the action of the generators is given by
  \begin{equation}\label{E:AkH}
H_1v^j_i=(k-j-2i)v^j_i,\quad H_2v^j_i=(i+j)v^j_i,
\end{equation}
  \begin{equation}\label{E:Akef1}
F_1v^j_i=v^j_{i+1},\quad E_2v^1_i=v^0_{i+1},
\end{equation}
  \begin{equation}\label{E:Akef2}
E_1v^j_{i}=[i][k-j+1-i]v^j_{i-1}, \quad F_2v^0_i=[i+1]v^1_{i-1}.
\end{equation}
\end{lemma}
\begin{proof}
 These relations come from the fact that $A_k$ is a highest weight module. In particular,  Equation~\eqref{E:AkH} comes from the known character of the
  $\mathfrak{sl}(2|1)$-module.  Equation~\eqref{E:Akef1} defines the basis because 
  that $E_2$ and $F_1$ commute.  Equation~\eqref{E:Akef2} can be deduced from
  the relations $E_1F_1-F_1E_1=\frac{K_1-K_1^{-1}}{q-q^{-1}}$ and
  $E_2F_2+F_2E_2=\frac{K_2-K_2^{-1}}{q-q^{-1}}$.
\end{proof}

Let $A=A_{\ell-1}$.

\begin{lemma}
  For any $\alpha\in\C\setminus\Z$, $0\le k\le \ell-2,0\le i\le\ell-1$,
  we have
  $$A\otimes V(\lambda^k_{\alpha+i})\simeq 
  \bar{\mathbb{C}} \otimes V(\lambda^{\ell-2-k}_{\alpha+i+k+1})
  \; \text{ and } \;
  A^{\otimes2}\otimes V(\lambda^k_{\alpha+i})\simeq \sigma(\bar 0, 1)\otimes
  V(\lambda^k_{\alpha+i})\; \text{ in
  }\mathcal{D}^{\aleph}.$$ 
  \end{lemma}
\begin{proof}
We will prove the first statement by strong induction on $k$. Then the second statement follows from applying the first statement twice. We need the following character formulas. We write $\chi^+$ and $\chi^-$ for the usual character and supercharacter, respectively.  
  In \cite{GP} we write (super)characters of
  $U_{q}\mathfrak{sl}(2|1)$-modules using variable
  $\textrm{e}^{\epsilon_1}$ and $\textrm{e}^{\epsilon_2}$. Let $x,y$
  be such that $\textrm{e}^{\epsilon_1}=y/x$ and
  $\textrm{e}^{\epsilon_2}=xy$. 
 Using $x$ and $y$ we can have the following (super)characters: 1) the character of the standard $(2|1)$-dimensional module $\mathsf{v}$ is   $\chi^{\pm}(\mathsf{v})=y(x+1/x\pm1)$.   The   (super)character of a typical module $V(\lambda^k_\alpha)$ is 
 $$\chi^{\pm}(V(\lambda^k_\alpha))=X_0^{\pm}y^{2\alpha+k}\dfrac{(x^{k+1}-x^{-k-1})}{x^{}-x^{-1}}$$  where $X_0^{\pm}=(1\pm\frac yx)(1\pm xy)$.  Finally, the (super)character for $A$ is 
 $$\chi^{\pm}(A)=\dfrac{y^n(x^{n+1}-x^{-n-1})\pm   y^{n+1}(x^{n}-x^{-n})}{x^{}-x^{-1}}.$$

For generic $q$ these character formula imply that $A\otimes V(\lambda^0_\alpha)$ is the direct sum of two simple modules of highest weights $(\ell-1,\alpha)$ and $(\ell-2,\alpha+1)$.  Then the supercharacter gives the parity of these modules and we can conclude $$A\otimes V(\lambda^0_\alpha)=V(\lambda^{\ell-1}_\alpha)\oplus ( \bar{\mathbb{C}}\otimes {V(\lambda^{\ell-2}_{\alpha+1})}).$$
  Since both modules appearing in this decomposition belong to
  the alcove (see \cite{AGP}) then the decomposition remains valid at
  root of unity. Finally, since $V(\lambda^{\ell-1}_\alpha)$ is a simple with $0$
  modified dimension, it is negligible in
  $\bar{\mathcal{D}}^{\wp}$ and thus zero in $\mathcal{D}^{\aleph}$.  Thus, this proves
  the $k=0$ statement, as $\alpha$ is generic.  
  
  Next, assume the statement is true for all natural number less than or equal to $k$.  The above character formulas show
  \begin{equation}\label{E:Vav}
  V(\lambda^k_\alpha)\otimes\mathsf{v}\simeq
  V(\lambda^{k+1}_\alpha)\oplus V(\lambda^{k-1}_{\alpha+1})\oplus ( \bar{\mathbb{C}}\otimes{V(\lambda^k_{\alpha+1})}).
  \end{equation}
  Combining this and the induction assumption we have
  $$A\otimes (V(\lambda^k_\alpha)\otimes\mathsf{v})\simeq 
  \left(A\otimes V(\lambda^{k+1}_\alpha)\right)\oplus
   ( \bar{\mathbb{C}}\otimes{V(\lambda^{\ell-k-1}_{\alpha+1+k})}\oplus{V(\lambda^{\ell-k-2}_{\alpha+k+2})}$$
 in $\mathcal{D}^{\aleph}$.  
On the other hand,  the induction assumption also shows
$$(A\otimes V(\lambda^k_\alpha))\otimes\mathsf{v}\simeq  ( \bar{\mathbb{C}}\otimes{V(\lambda^{\ell-k-2}_{\alpha+k+1})})\otimes\mathsf{v}$$
in $\mathcal{D}^{\aleph}$.  From Equation \eqref{E:Vav} the right hand side of the last equality is isomorphic to 
$$ ( \bar{\mathbb{C}}\otimes{V(\lambda^{\ell-k-1}_{\alpha+k+1})})\oplus  ( \bar{\mathbb{C}}\otimes{V(\lambda^{\ell-k-3}_{\alpha+k+2})})\oplus{V(\lambda^{\ell-k-2}_{\alpha+k+2})}$$ 
in $\mathcal{D}^{\aleph}$.  Thus by unicity of the decomposition into indecomposable, we have
$$A\otimes  V(\lambda^{k+1}_\alpha)\simeq ( \bar{\mathbb{C}}\otimes{V(\lambda^{\ell-k-3}_{\alpha+k+2})})$$ in $\mathcal{D}^{\aleph}$ and this completes the induction step.   
\end{proof}
\begin{lemma}
If $W$ is a simple object in $\mathcal{D}^{\aleph}_{\bar \alpha}$ then $S'(A,W) =q^{-2\bar \alpha \ell}$ where $S'$ is defined in Equation \eqref{eq:Sprime}.
\end{lemma}
\begin{proof}
The proof follows from a direct computation. A detailed presentation of an analogous
computation is given in \cite[Proposition 2.2]{GP2}.
\end{proof}
\begin{corollary}\label{rankintro}
 The rank of
  the $S$-matrix is atmost $\frac{\ell(\ell-1)}2$.
 Thus, Theorem \ref{degenmod} implies the Conjecture 5.19 of \cite{AGP} is false.
\end{corollary}
\begin{proof}
Let $W$ be a simple object in $\mathcal{D}^{\aleph}_{\bar \alpha}$.  Let $\alpha\in\C\setminus\Z$ be a lift of $\bar \alpha$.  Then for any  $0\le k\le \ell-2,0\le i\le\ell-1$, from the last two lemmas we have 
\begin{align*}
\label{}
  S'(V(\lambda^k_{\alpha+i}),W)  & =q^{2\bar \alpha \ell}S'(A,W)S'(V(\lambda^k_{\alpha+i}),W)  \\
    &  =q^{2\bar \alpha \ell}S'(A\otimes V(\lambda^k_{\alpha+i}),W)\\
    &=q^{2\bar \alpha \ell}S'(\bar{\mathbb{C}} \otimes V(\lambda^{\ell-2-k}_{\alpha+i+k+1}),W)\\
    &=-q^{2\bar \alpha \ell}S'( V(\lambda^{\ell-2-k}_{\alpha+i+k+1}),W).
\end{align*}
Since $-q^{2\bar \alpha \ell}$ does not depend on the simple module $W$ in $\mathcal{D}^{\aleph}_{\bar \alpha}$ then the lines in the $S$-matrix determined by $V(\lambda^k_{\alpha+i})$ and $V(\lambda^{\ell-2-k}_{\alpha+i+k+1})$ are proportional and the corollary follows.  
\end{proof}

We end the paper with the following questions:
\begin{itemize}
\item Is $A^{\otimes2}\simeq\sigma(\bar 1, 0)$ in $\mathcal{D}^{\aleph}$?
\item If so, is $\mathcal{D}^{\aleph}$ a relative
  modular category with translation group enriched with $A$?
\item If not, can we embbed $\mathcal{D}^{\aleph}$ in
  such a relative modular category?
  \item Is there a is a relative modular category with a translation group with $\sqrt{\sigma}$?
\end{itemize}

\end{document}

%% file: preamble.tex
\usepackage{amsmath,amsthm,amssymb,amsfonts,amscd, mathrsfs, mathtools, dsfont}
\usepackage[hyperindex,breaklinks]{hyperref}
\usepackage{cleveref}
\usepackage{tikz}
\usepackage{tikz-cd}
\usepackage{cancel}
\usepackage{fullpage}
\usepackage[parfill]{parskip}
\usepackage{enumerate}
\usepackage[shortlabels]{enumitem}
\usepackage{makecell}
\usetikzlibrary{matrix,arrows,decorations.pathmorphing}

\theoremstyle{definition}
\newtheorem{theorem}{Theorem}[section]
\newtheorem{corollary}[theorem]{Corollary}
\newtheorem{lemma}[theorem]{Lemma}
\newtheorem{proposition}[theorem]{Proposition}
\newtheorem{conjecture}[theorem]{Conjecture} 
 
\newtheorem{definition}[theorem]{Definition}

\newtheorem{remark}[theorem]{Remark}

\numberwithin{equation}{section}
\numberwithin{theorem}{section}

\definecolor{darkgreen}{rgb}{0,0.5,0}
\definecolor{darkblue}{rgb}{0,0.1,0.5}

\hypersetup{
    colorlinks=true, 
    linkcolor=darkblue, 
    urlcolor=blue, 
    linktoc=all, 
    citecolor=darkgreen
    }

\newcommand{\Id}{\mathrm{Id}}

\newcommand{\End}{\mathrm{End}}

\newcommand{\Hom}{\mathrm{Hom}}

\newcommand{\mcC}{\mathcal{C}}

\newcommand{\mcN}{\mathcal{N}}

